\documentclass[a4paper, 11pt]{amsart}  
\usepackage[setpagesize=false, hidelinks, bookmarks=true]{hyperref}
\usepackage{nameref}
\usepackage{amssymb}
\usepackage{amsmath,amsfonts, amsthm}
\usepackage{enumerate}
\usepackage{ascmac}
\usepackage[dvipdfmx]{graphicx}
\usepackage[dvipdfmx,svgnames]{xcolor}
\usepackage{tikz}
\usepackage{multicol}
\usepackage[arrow,matrix]{xy}
\usepackage{comment}
\usepackage{bm}

\usetikzlibrary{positioning, intersections, calc, arrows.meta,math}
\usetikzlibrary{patterns}
\usetikzlibrary{decorations.pathreplacing}
\usetikzlibrary{cd}

\usepackage{float}

\usepackage{pb-diagram}

\usepackage{multirow}
\usepackage[normalem]{ulem}
\useunder{\uline}{\ul}{}
\usepackage {diagbox}
\usepackage{ulem}

\makeatletter
\let\orgdescriptionlabel\descriptionlabel
\renewcommand*{\descriptionlabel}[1]{%
  \let\orglabel\label
  \let\label\@gobble
  \phantomsection
  \edef\@currentlabel{#1}%
  \let\label\orglabel
  \orgdescriptionlabel{#1}%
}
\makeatother

\numberwithin{equation}{section}

\newtheorem{thm}{Theorem}[section]
\newtheorem{prop}[thm]{Proposition}
\newtheorem{cor}[thm]{Corollary}
\newtheorem{lem}[thm]{Lemma}

\theoremstyle{definition}
\newtheorem{defn}[thm]{Definition}
\newtheorem{rmk}[thm]{Remark}

\newcommand{\Z}{\mathbb{Z}}
\newcommand{\N}{\mathbb{N}}
\newcommand{\Q}{\mathbb{Q}}
\newcommand{\F}{\mathbb{F}}

\begin{document} 

\title[Division properties of commuting polynomials]{Division properties of commuting polynomials} 
\author[K.~Hasegawa and R.~Sugiyama]{Kimiko Hasegawa and Rin Sugiyama}

\address[K.~Hasegawa]{Japan Women's university}
\email{m1716072hk@ug.jwu.ac.jp}
\address[R.~Sugiyama]{Japan Women's university}
\email{sugiyamar@fc.jwu.ac.jp}
\date{\today} 
\address{Department of Mathematics, Physics, and Computer science, Japan Women's University, 2-8-2 Mejirodai, Bunkyo-ku Tokyo, 112-8681 Japan}
%\email{}

\begin{abstract}
Polynomials that commute under composition are referred to  as \textit{commuting polynomials}.
In this paper, we study division properties for commuting polynomials with rational (and integer) coefficients.
As a consequence, we show an algebraic particularity of the commuting polynomials coming from weighted sums for cycle graphs with pendant edges (\cite{FHIK}).
We also discuss a set of commuting polynomials over a field of positive characteristic. 
\end{abstract}

\thanks{}
\maketitle
%\tableofcontents

%%%%%%%%%%%%%%%%%%%%%%%%%%%%%%%%%%%%
%         Intro
%%%%%%%%%%%%%%%%%%%%%%%%%%%%%%%%%%%%

\section{Introduction}
Polynomials commute under composition are referred to as \textit{commuting polynomials}.
%Such polynomials have appeared in a study of complex rational function by Julia \cite{Ju} and Ritt \cite{R} in 1920s.
A set of pairwise commuting polynomials containing one polynomial of each
positive degree is called a \textit{chain} (Definition \ref{chain}). 
The most easy and important example is the monic monomials $\{x^n\}_{n\in \N}$.
Another important example is given by the Chebyshev polynomials of the first kind $\{T_n(x)\}_{n\in\N}$ (Definition \ref{defn-T}).

Conjugation by a linear polynomial defines an equivalence relation called \textit{similarity}.
The similarity plays a key role in the study of commuting polynomials and chains. 
In 1950s, Block and Theilman \cite{BT}, Jacobstahl \cite{J} independently proved the classification theorem for chains (Theorem \ref{class-chain}), which states that up to similarity there are only two chains; one is the monic monomials.
The other is the Chebyshev polynomials.

Recently, Fujita-Hasegawa-Inaba-Kondo \cite{FHIK} gave two formulas for the weighted sums for cycle graphs with pendant edges, which use the following polynomials
\begin{align}\label{Fn}
F_n(x)=2T_n\left(\frac{x}2+1\right)-2,\qquad 
\tilde{F}_n(x)=(x+1)^n-1.
\end{align}
By direct computations, we easily see that the two sets $\{F_n(x)\}_{n\in\N}$ and $\{\tilde{F}_n(x)\}_{n\in\N}$ are also chains and similar to $\{T_n(x)\}_{n\in\N}$ and $\{x^n\}_{n\in\N}$ respectively.
Furthermore, the factorization of $F_n(x)$ in \cite{FHIK} for some small $n$ suggests a general form of factorization.
 
Our motivation comes from such particular algebraic properties of $F_n(x)$ and $\tilde{F}_n(x)$, and we will show that $F_n(x)$ and $\tilde{F}_n(x)$ are ``special" in all chains, with respect to some division properties (Corollary \ref{intro-consequence}).

We now give a more precise description of our work.

%%%%%%%%%%%%%%%%%%%%%%%%%%%%%%%%%%%%
% Intro -sub- Commuting polynomials
% %%%%%%%%%%%%%%%%%%%%%%%%%%%%%%%%%%%

\subsection{Commuting polynomials}
We recall several notions from \cite{Z} (\cite{BT} and \cite{W}).
%Let $A$ be an integral domain containing $\Z$ and $K$ be the quotient field of $A$.

\begin{defn}\label{chain}
Let $K$ be a commutative ring.
A \textit{chain} in $K[x]$ is a set of polynomials $\{f_n(x)\}_{n \in \N}\subset K[x]$ that satisfies deg$(f_n(x)) = n$ and any two polynomials in the set commute with each other. 
\end{defn}

\begin{rmk}
A chain of commuting polynomials is called an entire set of commutative polynomials in \cite{BT}.
\end{rmk}

\begin{defn}
Assume that $K$ is a field.
Let $f(x), g(x) \in K[x]$.
Then $f(x)$ is \textit{similar} to $g(x)$ if there exists $\lambda(x) = ax + b\in K[x]\ (a \neq 0)$ such that $g(x)=\lambda^{-1}( f( \lambda(x)))$.

Let $\{f_n(x)\}_{n\in \N}, \{g_n(x)\}_{n\in \N}$ be chains in $K[x]$.
Then $\{f_n(x)\}_{n\in \N}$ is \textit{similar} to $\{g_n(x)\}_{n\in \N}$ if and only if there exists $\lambda(x) = ax + b\in K[x]\ (a \neq 0)$ such that $g_n(x)=\lambda^{-1}( f_n( \lambda(x)))$ for any $n$.
\end{defn}

The following is the classification theorem for chains.

\begin{thm}[\cite{BT}, \cite{J}]\label{class-chain}
Assume that $K$ is a field of characteristic zero.
Let $\{f_n(x)\}_{n \in \N}$ be a chain in $K[x]$.
Then $\{f_n(x)\}_{n \in \N}$ is similar to only one of the two $\{x^n\}_{n \in \N}$ or $\{T_n(x)\}_{n \in \N}$.
\end{thm}

We discuss this theorem over a field of positive characteristic later.

From Theorem \ref{class-chain}, we use the following terminology in this paper.

\begin{defn}
A chain is said to be of \textit{monomial type} (resp. \textit{Chebyshev type}) if it is similar to $\{x^n\}_{n\in \N}$ (resp. $\{T_n(x)\}_{n\in \N}$).
\end{defn}

By the definition of similarity, a chain $\{f_n(x)\}_{n \in \N}$ of Chebyshev (resp. monomial) type in $\Q[x]$ is presented as follows; let $\lambda(x)=ax+b\ (a, b\in \Q, a\neq 0)$,
\begin{align*}
f_n(x) &=\lambda^{-1}( T_n(\lambda(x)))= \frac{1}{a} (T_n(ax + b) - b)\\
\Bigl( \text{resp. } f_n(x) &=\lambda^{-1}( \lambda(x)^n)= \frac{1}{a} ((ax + b)^n - b) \Bigr)
\end{align*}

%%%%%%%%%%%%%%%%%%%%%%%%%%%%%%%%%%%%
% Intro -sub- Results
% %%%%%%%%%%%%%%%%%%%%%%%%%%%%%%%%%%%

\subsection{Results}
We discuss the following conditions for chains of Chebyshev (resp. monomial) type in $\Q[x]$ separately.
\begin{defn}\label{conditions}
We define conditions (i)--(iv) for a chain $\{f_n(x)\}_{n \in \N}$ in $\Q[x]$ as follows;
\begin{description}
\item[Condition (i)\label{itm:i}]
For all $n \in \N$, $f_n(x)$ is a monic polynomial  
\item[Condition (ii)\label{itm:ii}]
For all $n \in \N$, $f_n(x) \in \Z[x]$  
\item[Condition (iii)\label{itm:iii}]
For all $m, n \in \N, $ $m \mid n \Leftrightarrow f_m(x) \mid f_n(x)$ in $\Q[x]$ 
\item[Condition (iv)\label{itm:iv}] 
$\gcd(f_m(x), f_n(x)) = f_{\gcd(m, n)}(x)$  
\end{description}
\end{defn}

 We interpret Conditions (i)--(iii) in terms of $\lambda(x)=ax+b$, i.e., conditions on $a$ and $b$.
 In particular, for \ref{itm:iii}, we first consider the Euclidean division of $f_n(x)$ by $f_m(x)$ and give a description of the quotient and the remainder (Proposition \ref{E-Div-T} and Proposition \ref{E-Div-mono}).
Using the description of the remainder, we discuss the condition on $a, b$ corresponding to \ref{itm:iii} (Corollary \ref{Div-T} and Corollary \ref{Div-mono}).
For \ref{itm:iv}, we first give the factorization in $\Z[x]$ (Theorem \ref{factorization-T}) and using it we discuss a property on the greatest common divisor (Corollary \ref{gcd-T} and Corollary \ref{gcd-mono}).

As a consequence of results in this paper, we have the following corollary which shows an algebraic particularity of the two chains $\{F_n(x)\}$ and $\{\tilde{F}_n(x)\}$ of \eqref{Fn} in all chains.

\begin{cor}\label{intro-consequence}
Suppose a chain $\{f_n(x)\}_{n \in \N}$ satisfies all of \ref{itm:i}, \ref{itm:ii} and \ref{itm:iii}.
Then we have
\[
f_n(x)=F_n(x)=2T_n\left(\frac{x}2+1\right)-2\quad \text{or}\quad f_n(x)=\tilde{F}_n(x)=(x+1)^n-1.
\]
Additionally, in this case \ref{itm:iv} is also satisfied. 
\end{cor}

\begin{rmk}
For any sequence of polynomials, we can ask wether Conditions (i)--(iv) hold or not.
For example, Webb and Parberry \cite{WP} proved that Fibonacci polynomials $\{U_n(x)\}_{n\in\N}$ satisfy Conditions (ii)--(iv).
Condition (i) holds except $U_0=0, U_1=1$ .
But $\{U_n(x)\}_{n\in\N}$ is not chain.
It is natural to ask how much Condition (i)--(iv) characterize a sequence $\{f_n(x)\}_{n\in \N}$ of polynomials with $\deg f_n(x)=n$ but not a chain.
\end{rmk}

Let $k$ be a field of characteristic $p>0$.
We also discuss chains in $k[x]$ and prove the following classification theorem. 
\begin{thm}[Theorem \ref{thm-A}]\label{intro-thm-A}
%Assume that $p\geq 3$.
Let $\{f_n(x)\}_{n \in \N}$ be a chain in $k[x]$.
Then $\{f_n(x)\}_{n \in \N}$ is similar to only one of the two $\{ x^n\}_{n\in\N}$ or $\{G_n(x)\}_{n\in\N}$.
Here $G_n(x):=F_n(x) \pmod p$.
\end{thm}

We also discuss the factorization of $F_n(x)$ and $\tilde{F}_n(x)\pmod p$ and show the following.
\begin{cor}[Corollary \ref{cor-A}]\label{intro-cor-A}
Let $p$ be a prime.
For any positive integer $r$, we have
\[
F_{p^r}(x)\equiv \tilde{F}_{p^r}(x)\equiv x^{p^r}\pmod p.
\]
\end{cor}

Lastly, since polynomials $F_n(x), \tilde{F}_n(x)$ come from the weighted sums for cycle graphs with pendant edges, it is natural to ask a graph theoretical explanation of the division properties and the commutativity under composition.
Such graph theoretical approach to these properties will be developed in a subsequent paper.

\bigskip
This paper is organized as follows.
We discuss conditions in Definition \ref{conditions} for chains of Chebyshev type in \S2, for chains monomial type in \S3.
%We state (again) the consequence of our results in this paper in \S4.
In \S4 we discuss chains over a field of positive characteristic.

%%%%%%%%%%%%%%%%%%%%%%%%%%%%%%%%%%%%
%           Acknowledgement
%%%%%%%%%%%%%%%%%%%%%%%%%%%%%%%%%%%%

\bigskip\noindent
\textbf{Acknowledgement.}
The authors are grateful to Hajime Fujita and Takefumi Kondo for fruitful discussions and valuable comments for this work, and to Daisuke Shiomi for valuable comments to an earlier draft of this paper.

%%%%%%%%%%%%%%%%%%%%%%%%%%%%%%%%%%%%
%           Chains of Chebyshev type
%%%%%%%%%%%%%%%%%%%%%%%%%%%%%%%%%%%%

\section{Chains of Chebyshev type}
We first recall the definition and some properties of Chebyshev polynomials from \cite{RTW}.
From \S2.2, we discuss conditions in Definition \ref{conditions} for a chain of Chebyshev type in $\Q[x]$

%%%%%%%%%%%%%%%%%%%%%%%%%%%%%%%%%%%%
%           Chebyshev polynomials
%%%%%%%%%%%%%%%%%%%%%%%%%%%%%%%%%%%%

\subsection{Chebyshev polynomials}

\begin{defn}\label{defn-T}
The {\it Chebyshev polynomial of the first kind} $T_n(x)$ is defined by the following recurrence relation;
\begin{align}\label{rec-rel-T}
T_0(x)=1,\quad
T_1(x)=x,\quad 
T_n(x)=2xT_{n-1}(x)-T_{n-2}(x)\ (n\geq 2).
\end{align}

The {\it Chebyshev polynomial of the second kind} $U_n(x)$ is defined by the following recurrence relation;
\begin{align}\label{rec-rel-U}
U_0(x)=1,\quad
U_1(x)=2x,\quad 
U_n(x)=2xU_{n-1}(x)-U_{n-2}(x)\ (n\geq 2).
\end{align}
\begin{rmk}
The Chebyshev polynomials $T_n(x)$ and $U_n(x)$ also may be defined as follows;
\begin{align*}
T_n(x)=\cos(n\arccos x),\qquad U_n(x)=\frac1{n+1}T'_{n+1}(x)=\frac{\sin((n+1)\arccos x)}{\sin(\arccos x)}
\end{align*}
\end{rmk}
\end{defn}

\begin{prop}[\text{\cite[p.6, Property 2]{RTW}}]\label{E-Div}
For any $m, n\, (m\leq n)$, the quotient $q_{m,n}(x)$ and the remainder $r_{m,n}(x)$ of the Euclidean division of $T_n(x)$ by $T_m(x)$ are given by
\begin{align*}
q_{m,n}(x)=2\sum_{i=1}^l(-1)^iT_{n-(2i-1)m}(x), \quad 
r_{m,n}(x)=(-1)^lT_{|n-2lm|}(x),
\end{align*}
if there is an integer $l\geq 1$ satisfying $(2l-1)m<n\leq 2lm$, otherwise
\begin{align*}
q_{m,n}(x)=2\sum_{i=1}^{l-1}(-1)^iT_{n-(2i-1)m}(x)+(-1)^{l-1}, \quad 
r_{m,n}(x)=0,
\end{align*}
where $l$ satisfies $n=(2l-1)m$.
\end{prop}

%%%%%%%%%%%%%%%%%%%%%%%%%%%%%%%%%%%%
%          Monic and integer coefficient conditions  for T
%%%%%%%%%%%%%%%%%%%%%%%%%%%%%%%%%%%%

\subsection{Monic and integer coefficient conditions}

We discuss conditions in Definition \ref{conditions} for a chain $\{f_n(x)\}_{n \in \N}$ of Chebyshev type in $\Q[x]$;
$$
f_n(x) = \frac{1}{a} (T_n(ax + b) - b) \qquad (a, b\in \Q, a\neq 0).
$$

We prepare a lemma about the recurrence relation for $\{f_n(x)\}_{n \in \N}$.
\begin{lem}\label{rec-rel}
Let $\{f_n(x)\}_{n \in \N}$ be a chain of Chebyshev type in $\Q[x]$. 
Then $f_n(x)$ has the following recurrence relation for $n\geq 3$
\[
f_n(x) 
= 2(ax + b)\cdot f_{n-1}(x) - f_{n-2}(x) + 2bx + \frac{2b(b - 1)}{a}.
\]
\end{lem}
\begin{proof}
By Proposition \ref{rec-rel-T}, we have
\begin{align*}
f_n(x)
&=\frac{1}{a}T_n(ax + b) - \frac ba\\
&=\frac{1}{a}(2(ax+b)T_{n-1}(ax+b)-T_{n-2}(ax+b)) - \frac ba\\
&=2(ax+b)f_{n-1}(x)+\frac{2b(ax+b)}a-f_{n-2}(x)-\frac{2b}2\\
&=2(ax + b)\cdot f_{n-1}(x) - f_{n-2}(x) + 2bx + \frac{2b(b - 1)}{a}.
\end{align*}
\end{proof}

\begin{prop}\label{Z[x]-T}
Let $\{f_n(x)\}_{n \in \N}$ be a chain of Chebyshev type in $\Q[x]$, that is, $f_n(x) = \frac{1}{a} (T_n(ax + b) - b)$. 
Then we have
\begin{enumerate}[{\rm (1)}]
\item
\ref{itm:i} holds if and only if $a = \frac{1}{2}$.
\item
\ref{itm:ii} holds if and only if  $a\in \frac12 \Z\ \text{and}\ b-1\in a\Z$. 
\end{enumerate}
\end{prop}

\begin{proof}
(1)
It is clear because the leading coefficient of $T_n(x)$ is $2^{n-1}$.

(2)
Since $f_2(x) = 2ax^2 + 4bx + \frac{(2b + 1)(b - 1)}{a}$ and $f_3(x)=4a^2x^3+12abx^2+3(4b^2-1)x+\frac{4(b^3-b)}a$, by an elementary computation, we have
\begin{align*}
f_2(x), f_3(x)\in \Z[x]\ \iff \  a\in \frac12 \Z\ \text{and}\ b-1\in a\Z.
\end{align*}
In this case, we also have
\[
2(ax + b),\ 2bx + \frac{2b(b - 1)}{a} \in \Z[x].
\]
From the recurrence relation (Lemma \ref{rec-rel})
\begin{align*}
f_n(x) 
= 2(ax + b)\cdot f_{n-1}(x) - f_{n-2}(x) + 2bx + \frac{2b(b - 1)}{a},
\end{align*}
the desired statement is obtained by induction on $n$.
\end{proof}

%%%%%%%%%%%%%%%%%%%%%%%%%%%%%%%%%%%%
%          Euclidean division for  T
%%%%%%%%%%%%%%%%%%%%%%%%%%%%%%%%%%%%

\subsection{Euclidean division}
\begin{prop}\label{E-Div-T}
For $a, b\in \Q\, (a\neq 0)$, let $f_n(x) = \frac{1}{a} (T_n(ax + b) - b)$.
For any positive integers $m, n\, (m< n)$, the quotient $q_{m,n}(x)$ and the remainder $r_{m,n}(x)$ of the Euclidean division of $f_n(x)$ by $f_m(x)$ are given as follows; put $k=[\frac nm]$,
\begin{align}
q_{m,n}(x)=2a\sum_{i=1}^kc_if_{n-im}(x)+d_{m,n},
\end{align}
\begin{align}
r_{m,n}(x)=
\begin{cases}
\dfrac 1a (bc_k-b-c_{k-1}) &\text{if}\ \ n=km \bigskip\\
c_{k+1}f_{n-km}(x)-c_kf_{(k+1)m-n}(x)\medskip\\\qquad \qquad +\dfrac ba(c_{k+1}-c_{k}-1)& \text{if}\ \ m\nmid n
\end{cases}
\end{align}
Here $c_i=U_{i-1}(b)\in \Z[b]$ is the Chebyshev polynomial of the second kind % and $c_i=U_{i-1}(b)\in \Z[b]$
% is defined as %by the following recurrence relation
% \[
% c_0=0, \quad c_i=U_{i-1}(b) \ (i\geq 1),%c_1=1,\quad c_{i+1}=2bc_i-c_{i-1}\ (i\geq 1),
% \]
and $d_{m,n}\in \Z[b]$ is defined as
\begin{align*}
d_{m,n}=
\begin{cases}
\displaystyle 2b\sum_{i=1}^kc_i-c_k& \text{if}\ \ n=km \bigskip\\
\displaystyle 2b\sum_{i=1}^{k}c_i& \text{if}\ \ m\nmid n 
\end{cases}
\end{align*}
\end{prop}

\begin{proof}
Put $g_n(x)=T_n(x)-b$, and let $Q_{m,n}(x)$ and $R_{m,n}(x)$ be the quotient and  the remainder of the Euclidean division of $g_n(x)$ by $g_m(x)$.
Then we have
\[
q_{m,n}(x)=Q_{m,n}(ax+b),\quad r_{m,n}(x)=\frac1a R_{m,n}(ax+b)
\]
since $f_n(x)=\frac1a g_n(ax+b)$ and $g_n(x)=af_n(\frac{x-b}a)$.
Thus our task is to compute $Q_{m,n}(x)$ and $R_{m,n}(x)$.

From Proposition \ref{E-Div} we have 
$$
T_n(x)=2T_{n-m}(x)T_m(x)-T_{|n-2m|}(x) \qquad (n\geq m)
$$
and hence
\begin{equation}\label{eq1}
g_n(x)=2(g_{n-m}(x)+b)g_m(x)+2bg_{n-m}(x)-g_{|n-2m|}(x)+2b(b-1).
\end{equation}
Using this relation and induction on $k$, we will prove that
\begin{align*}
Q_{m,n}(x)&=2\sum_{i=1}^kc_ig_{n-im}(x)+d_{m,n},\\
R_{m,n}(x)
&=
\begin{cases}
b(c_k-1)-c_{k-1}&\text{if}\ n=km\medskip\\
c_{k+1}g_{n-km}(x)-c_kg_{(k+1)m-n}(x)\\
\qquad \qquad +b(c_{k+1}-c_k-1)& \text{if}\ m\nmid n.
\end{cases}
\end{align*}

%The statement for $k=1$ is clear from \eqref{eq1}.
We first prove in case $n=2m$.
By \eqref{eq1} we have
\begin{align*}
g_{2m}(x)
&=2(g_{m}(x)+b)g_m(x)+2bg_m(x)-g_0(x)+2b(b-1)\\
&=(2g_m+4b)g_m+2b^2-b-1
\end{align*}
which shows the statement for $n=2m.$

We assume that $n=km\ (k>2)$. From \eqref{eq1}, the induction hypothesis and the recurrence relation \eqref{rec-rel-U}, we obtain that
\begin{align*}
Q_{m,km}(x)
&=2g_{(k-1)m}(x)+2b+2bQ_{m,(k-1)m}-Q_{m,(k-2)m}\\
&=2\left( g_{(k-1)m}(x)+\sum_{i=1}^{k-1}(2bc_i-c_{i-1})g_{(k-i-1)m}(x)\right)\\&\qquad \qquad +2bd_{m,(k-1)m}-d_{m,(k-2)m}+2b\\
&=2\sum_{i=1}^kc_ig_{(k-i)m}(x)+2b\sum_{i=1}^kc_i-c_k
\end{align*}
and
\begin{align*}
R_{m,km}(x)
&=2bR_{m,(k-1)m}-R_{m,(k-2)m}+2b(b-1)\\
&=2b^2(c_{k-1}-1)-2bc_{k-2}-b(c_{k-2}-1)+c_{k-3}+2b(b-1)\\
&=bc_k-b-c_{k-1}.
\end{align*}

Next we consider the case $m\nmid n$.
If $k=2$, by \eqref{eq1}, we have
\begin{align*}
g_n(x)
&=2(g_{n-m}(x)+b)g_m(x)+2bg_{n-m}(x)-g_{|n-2m|}(x)+2b(b-1)\\
&=(2g_{n-m}(x)+4bg_{n-2m}+2b(2b+1))g_m(x)\\
&\qquad \qquad +(4b^2-1)g_{n-2m}-2bg_{3m-n}(x)+b(4b^2-2b-2)
\end{align*}
and hence the statement holds for $k=2$.
We assume that $k>2$. By the same argument as above, we obtain that
\begin{align*}
Q_{m,n}(x)
&=2g_{n-m}(x)+2b+2bQ_{m,n-m}-Q_{m,n-2m}\\
&=2\left( g_{n-m}(x)+\sum_{i=1}^{k-1}(2bc_i-c_{i-1})g_{n-(i+1)m}(x)\right)\\&\qquad \qquad +2b\sum_{i=1}^{k-1}2bc_{i}-2b\sum_{i=1}^{k-2}c_{i}+2b\\
&=2\sum_{i=1}^kc_ig_{n-im}(x)+2b\sum_{i=1}^{k}c_{i}
\end{align*}
and
\begin{align*}
R_{m,n}(x)
&=2bR_{m,n-m}-R_{m,n-2m}+2b(b-1)\\
&=(2bc_k-c_{k-1})g_{n-km}(x)-(2bc_{k-1}-c_{k-2})g_{(k+1)m-n}(x)\\
&\hspace{3cm}+b(2bc_{k}-2bc_{k-1}-c_{k-1}+c_{k-2}-1)\\
&=c_{k+1}g_{n-km}(x)-c_kg_{(k+1)m-n}(x)+b(c_{k+1}-c_k-1).
\end{align*}
Now the statement is proved.
\end{proof}

Next, we give an explicit formula for $c_i=U_{i-1}(b)$.
% By the recurrence relation \eqref{rec-rel-U} for $U_n(x)$, one can easily obtain the following representation
% \[
% U_n(x)=\frac{\left(x+\sqrt{x^2-1}\right)^{n+1}-\left(x-\sqrt{x^2-1}\right)^{n+1}}{2\sqrt{x^2-1}}.
% \]
The following lemma follows from \eqref{rec-rel-U} by induction.

\begin{lem}\label{lem-Ck}
For $b\in \Q$, an explicit formula of $c_k=U_{k-1}(b)$ is given as follows
\begin{enumerate}[{\rm (1)}]
\item
%in case $b=1$, $c_k=k$
$U_{k-1}(1)=k$,\quad
%in case $b=-1$, $c_k=(-1)^{k-1}k$
$U_{k-1}(-1)=(-1)^{k-1}k$

\item
%in case $b=0$,
$
U_{k-1}(0)
=
\begin{cases}
0&(k\equiv 0 \mod 2)\\
1&(k\equiv 1 \mod 4)\\
-1&(k\equiv 3 \mod 4)
\end{cases}
$
\item
%in case $b=\frac12$, 
$
U_{k-1}(\frac12)
=
\begin{cases}
0&(k\equiv 0 \mod 3)\\
1&(k\equiv 1,2 \mod 6)\\
-1&(k\equiv 4,5 \mod 6)
\end{cases}
$,\quad 
% \item
% in case $b=-\frac12$, 
$
U_{k-1}(-\frac12)
=
\begin{cases}
0&(k\equiv 0 \mod 3)\\
1&(k\equiv 1 \mod 3)\\
-1&(k\equiv 2 \mod 3)
\end{cases}
$
\end{enumerate}
In case $b\neq \pm1$, we have
\[
U_{k-1}(b)
=
\frac{\alpha^k-\alpha^{-k}}{\alpha-\alpha^{-1}}
\]
Here $\alpha=b+\sqrt{b^2-1}$.
\end{lem}

\begin{cor}\label{Div-T}
Let $f_n(x)$ be as in Proposition \ref{E-Div-T}.
If $b\neq 0,\pm1, \pm \frac12$, then $f_m(x)$ is not a divisor of $f_n(x)$ for any $m<n.$
In case $b=0, \pm 1, \pm \frac12$, a condition that $f_m(x)$ is a divisor of $f_n(x)$ is equivalent to
the following
\begin{enumerate}[{\rm (1)}]
\item
in case $b=0$, $m\mid n$ and $\gcd(n/m, 2)=1$.
\item
in case $b=1$, $m\mid n$.
\item
in case $b=-1$, $m\mid n$ and $\gcd(n/m, 2)=1$.
\item
in case $b=\frac12$, $m\mid n$ and $\gcd(n/m, 6)=1$.
\item
in case $b=-\frac12$, $m\mid n$ and $\gcd(n/m, 3)=1$.
\end{enumerate}
\end{cor}

\begin{proof}
The condition that $f_m(x)$ is a divisor of $f_n(x)$ is equivalent to that the remainder $r_{m,n}(x)$ is zero.
By Proposition \ref{E-Div-T}, if $r_{m,n}(x)$ is zero for $m\nmid n$ then we have
\begin{enumerate}[(i)]
\item
in case $n-km\neq (k+1)m-n$, $c_{k+1}=c_k=0$ and $b=0$.
\item
in case $n-km=(k+1)m-n$, $c_{k+1}=c_k$ and $b=0$.
\end{enumerate}
But, if $b=0$ then both of conditions (i) and (ii) on $c_k$ can not be happen by Lemma \ref{lem-Ck}.
Thus $r_{m,n}(x)$ is not zero for any $m\nmid n$.

Let $n=km$.
If $b\neq \pm1$ then we have by Proposition \ref{E-Div-T} and Lemma \ref{lem-Ck}, 
\begin{align*}
r_{m,n}(x)=0
&\iff 
b(c_k-1)-c_{k-1}=0\\
&\iff 
2b(c_k-1)-2c_{k-1}=0\\
&\iff 
(\alpha+\alpha^{-1})\left(\frac{\alpha^k-\alpha^{-k}}{\alpha-\alpha^{-1}}-1\right)-2\left(\dfrac{\alpha^{k-1}-\alpha^{-k+1}}{\alpha-\alpha^{-1}}\right)=0\\
&\iff 
(\alpha^{k+1}-1)(\alpha^{k-1}-1)(\alpha^2-1)=0
\end{align*}
where $\alpha=b+\sqrt{b^2-1}$.
Thus if $r_{m,n}(x)=0$ then $\alpha$ is a root of unity.
On the other hand, the roots of unity in a quadratic extension of $\Q$ are 
\[
\pm1, \quad \pm \sqrt{-1}, \quad\pm \frac12\pm \frac{\sqrt{-3}}2, \quad \pm \frac12\mp \frac{\sqrt{-3}}2.
\]
Hence if $b\neq 0,\pm 1, \pm\frac12$ then $\alpha=b+\sqrt{b^2-1}$ cannot be a root of unity and $r_{m,n}(x)$ is not zero, which proves the first statement of this corollary.
In case $b=0,\pm 1, \pm\frac12$, by Lemma \ref{lem-Ck} again one can easily prove that the condition $b(c_k-1)-c_{k-1}=0$ is equivalent to the desired one.
\end{proof}

\begin{rmk}
In case $a=1$ and $b=0$, Proposition \ref{E-Div-T} and Corollary \ref{Div-T} have proved in \cite{RTW}.
\end{rmk}

%%%%%%%%%%%%%%%%%%%%%%%%%%%%%%%%%%%%
%          GCD for T
%%%%%%%%%%%%%%%%%%%%%%%%%%%%%%%%%%%%

\subsection{Greatest common divisor and Factorization}
From Proposition \ref{Z[x]-T}, a chain of monic polynomials of the Chebyshev type in $\Z[x]$ is given by
\[
\left\{2T_n\left(\frac{x}2+b\right)-2b\right\}_{n\in \N}\quad \left(b\in \frac12\Z\right).
\]
For $b=0,\pm1, \pm\frac12$, we describe all roots of $T_n(\frac x2)-b$ and give the factorization of $2T_n\left(\frac{x}2+b\right)-2b$.
Using this, we show a property for the greatest common divisor in the case of Chebyshev type.

\begin{lem}\label{lem}
For any $n>0$, we have
$T_n\left(\dfrac{t + t^{-1}}{2}\right) = \dfrac{t^n + t^{-n}}{2}. $
\end{lem}

\begin{proof}
We prove by induction on $n$.
For $n=1,2$, we have
$$T_1\left(\frac{t + t^{-1}}{2}\right) = \frac{t + t^{-1}}{2}, \quad T_2\left(\frac{t + t^{-1}}{2}\right) = 2 \cdot \left(\frac{t + t^{-1}}{2}\right)^2 - 1 = \frac{t^2 + t^{-2}}{2}. $$
For $n \geq 3$, by the recurrence relation of $T_n(x)$ and the induction hypothesis, we obtain that
\begin{align*}
T_n\left(\frac{t + t^{-1}}{2}\right) &= 2 \cdot \left(\frac{t + t^{-1}}{2}\right) T_{n-1}\left(\frac{t + t^{-1}}{2}\right) - T_{n-2}\left(\frac{t + t^{-1}}{2}\right) \\
& = (t + t^{-1}) \cdot \frac{t^{n-1} + t^{-(n-1)}}{2} - \frac{t^{n-2} + t^{-(n-2)}}{2} \\
& = \frac{t^{n} + t^{-n}}{2}, 
\end{align*}
which completes the proof.
\end{proof}

\begin{prop}\label{root-T}
For $b\in \frac12\Z$, the roots of $T_n(\frac x2) - b$ are given as follows;
\begin{enumerate}[{\rm (1)}]
\item
in case $|b|\leq 1$, $x=2\cos(\frac{\theta_b+2i\pi}n)$\ $(i=0,\dots,n-1)$. Here $\theta_b={\rm Arccos}(b).$
In particular, for $b=\pm 1$, the roots not equal to $\pm 2$ are multiple roots.\medskip

\item
in case $b>1$, $x=\sqrt[n]{b + \sqrt{b^2 -1 }}\zeta_n^i + \sqrt[n]{b - \sqrt{b^2 -1 }}\zeta_n^{-i}\ (i=0,1\dots,n-1)$\medskip
\item
in case $b<-1$, $x=\sqrt[n]{|b + \sqrt{b^2 -1}|}\zeta_{2n}^{2i+1} + \sqrt[n]{|b - \sqrt{b^2 -1}|}\zeta_{2n}^{-2i-1}\ (i=0,1\dots,n-1)$\end{enumerate}
\end{prop}

\begin{proof}
We first prove (1).
Put $x=2\cos t$.
Then we have
\[
T_n\left(\frac x2 \right) -b=T_n(\cos t)-b=\cos (nt)-\cos\theta_b=0.
\]
Solving the last equation about $t$, we get
\[
t=\pm \frac{\theta_b+2i\pi}n\quad (i=0,\dots,n-1)
\]
and thus the roots are given by
\[
x=2\cos\left(\frac{\theta_b+2i\pi}n\right)\quad (i=0,\dots,n-1).
\]
Furthermore,  since $\theta_1=0$ and $\theta_{-1}=\pi$, the roots for $b=1$ and $b=-1$ are of the form respectively
$$
2\cos\left(\frac{2i\pi}n\right)\quad\text{and}\quad  2\cos\left(\frac{(2i+1)\pi}n\right) \quad (i=0,\dots,n-1).
$$
On the other hand, we have 
$$
\frac{d}{dx}\left(T_n\left(\frac{x}2\right)-b\right)=\frac n2 U_{n-1}\left(\frac{x}2\right).
$$
Since it is know that the roots of $U_{n-1}(x)$ are given by
\[
\cos\frac{i\pi}{n}\quad (i=1,\dots, n-1),
\]
we obtain that for $i=1,\dots, n-1$ except $i=k$ when $n=2k$, 
\[
\frac{d}{dx}\left(T_n\left(\frac{x}2\right)-1\right)\Big|_{x=2\cos(\frac{2i\pi}n)}=\frac n2 U_{n-1}\left(\cos\frac{2i\pi}n\right)=0
\]
and for $i=1,\dots, n-1$ except $i=k$ when $n=2k+1$,
\[
\frac{d}{dx}\left(T_n\left(\frac{x}2\right)+1\right)\Big|_{x=2\cos(\frac{(2i+1)\pi}n)}=\frac n2 U_{n-1}\left(\cos\frac{(2i+1)\pi}n\right)=0.
\]
Hence in cases $b=\pm1$ the roots not equal to $\pm 2$ are multiple roots, and the proof for (1) completes.

We next  put $x= t + t^{-1}$. Then Lemma \ref{lem}, we have
$$
T_n\left(\frac x2 \right) -b= T_n\left(\frac{t + t^{-1}}{2}\right)-b = \frac{t^n + t^{-n}}{2} -b=0.
$$
Solving the last equation about $t$, we get
\[
 t = 
 \begin{cases}
 \sqrt[n]{b \pm \sqrt{b^2 - 1}}\zeta_n^i& (b>1)\bigskip\\
 \sqrt[n]{|b \pm \sqrt{b^2 - 1}|}\zeta_{2n}^{2i+1}&(b<-1)
 \end{cases}
 \quad (i=0,\dots,n-1)
\]
which shows (2) and (3).
\end{proof}

To give the factorization of $f_n(x)$, we prepare some notation.
\begin{defn}\label{def-c}
For any positive integer $n$, we write $\Psi_n(x)$ for the minimal polynomial of $2\cos(\frac{2\pi}n)$, and put
\begin{align*}
c_n(x)&:=
\begin{cases}
x& (n : odd)\\
x(x + 4) & (n : even).
\end{cases}\\
c^*_n(x)&:=
\begin{cases}
x& (n : odd)\\
1& (n : even).
\end{cases}
\end{align*}
\end{defn}

\begin{thm}\label{factorization-T}
For $b=0,\pm 1, \pm\frac12$, the factorization of $f_n(x)=2T_n\left(\frac{x}2+b\right)-2b$ in $\Z[x]$ is given as follows;
\begin{enumerate}[{\rm (1)}]
\item
in case $b=0$, 
$\displaystyle
f_n(x)=\prod_{\substack{k \mid n \\ \gcd(n/k, 2) = 1}} \Psi_{4k}(x).
$\medskip
\item
in case $b=1$, 
$\displaystyle
f_n(x)
=c_n(x)\prod_{\substack{k \mid n,\ 2 < k}} \Psi_k(x+2)^2.
$\medskip

\item
in case $b=-1$,
$\displaystyle
f_n(x)=c^*_n(x)\prod_{\substack{k \mid n,\ 1<k \\ \gcd(n/k, 2) = 1}} \Psi_{2k}(x-2)^2.
$\medskip

\item
in case $b=\frac12$,
$\displaystyle
f_n(x)=
\prod_{\substack{k \mid n \\ \gcd(n/k, 6) = 1}} \Psi_{6k}(x+1).
$\medskip

\item
in case $b=-\frac12$,
$\displaystyle
f_n(x)=
\prod_{\substack{k \mid n \\ \gcd(n/k, 3) = 1}} \Psi_{3k}(x-1).$
\end{enumerate}
\end{thm}

\begin{proof}
Since $2\cos(\frac{2\pi}n)$ is an algebraic integer, the minimal polynomial $\Psi_n(x)$ is a monic irreducible polynomial in $\Z[x]$.
Also it is know that the degree of $\Psi_n(x)$ is $\frac{\phi(n)}2$.
Here $\phi(-)$ is the Euler's totient function.
Thus a polynomial of the form $\Psi_n(x+b)$ with $b\in \Z$ is a monic irreducible polynomial of degree $\frac{\phi(n)}2$.

Since $f_n(x)=2(T_n(\frac{x+2b}2)-b)$, by Proposition \ref{root-T}, the right hand side of the desired factorization is a divisor of $f_n(x)$.
Since the both side of the factorization are monic, it is now enough to show that the degree of the right hand side is $n$.
\begin{enumerate}
\item
We write $n=2^em$ with odd $m$. 
\begin{align*}
\deg \left(\ \prod_{\substack{k \mid n \\ \gcd(n/k, 2) = 1}} \Psi_{4k}(x)\right)
&=\sum_{l \mid m}\deg \Psi_{2^{e+2}l}(x)\\
&=\sum_{l \mid m} \frac{\phi(2^{e+2}l)}2\\
&=\sum_{l \mid m}2^e\phi(l)
=2^e\sum_{l \mid m} \phi(l)=2^e m=n
\end{align*}
\item
\begin{align*}
\deg \left(c_n(x) \prod_{\substack{k \mid n \\ 2 < k}} \Psi_k(x+2)^2 \right)
&= \deg(c_n(x)) + \sum_{\substack{k \mid n \\ 2 < k}}2 \deg(\Psi_k(x+2)) \\
&= \deg(c_n(x)) + \sum_{\substack{k \mid n \\ 2 < k}}2 \cdot \frac{\phi(k)}{2} 
= \sum_{k \mid n}\phi(k) = n. 
\end{align*}
\item
We write $n=2^em$ with odd $m$. 
If $e=0$ then we have
\begin{align*}
\deg \left(c^*_n(x)\prod_{\substack{k \mid n,\ 1<k \\ \gcd(n/k, 2) = 1}} \Psi_{2k}(x-2)^2\right)
&=\deg c^*_n(x)+\sum_{l \mid m,\ 1<l}2\deg \Psi_{2l}(x-2)\\
&=1+\sum_{l \mid m,\ 1<l} \phi(2l)\\
&=\sum_{l \mid m} \phi(l)=m=n.
\end{align*}
If $e>0$ then we have
\begin{align*}
\deg \left(c^*_n(x)\prod_{\substack{k \mid n,\ 1<k \\ \gcd(n/k, 2) = 1}} \Psi_{2k}(x-2)^2\right)
&=\deg c^*_n(x)+\sum_{l \mid m}2\deg \Psi_{2^{e+1}l}(x-2)\\
&=\sum_{l \mid m} \phi(2^{e+1}l)\\
&=2^e\sum_{l \mid m} \phi(l)=2^e m=n.
\end{align*}
\item
We write $n=2^{e_1}3^{e_2}m$ with $\gcd(6,m)=1$. 
\begin{align*}
\deg \left(\ \prod_{\substack{k \mid n \\ \gcd(n/k, 6) = 1}} \Psi_{6k}(x+1)\right)
&=\sum_{l \mid m}\deg \Psi_{2^{e_1+1}3^{e_2+1}l}(x+1)\\
&=\sum_{l \mid m} \frac{\phi(2^{e_1+1}3^{e_2+1}l)}2\\
&=2^{e_1}3^{e_2}\sum_{l \mid m}\phi(l)
=2^{e_1}3^{e_2} m=n.
\end{align*}

\item
We write $n=3^em$ with $\gcd(3,m)=1$. 
\begin{align*}
\deg \left(\ \prod_{\substack{k \mid n \\ \gcd(n/k, 3) = 1}} \Psi_{3k}(x-1)\right)
&=\sum_{l \mid m}\deg \Psi_{3^{e+1}l}(x-1)\\
&=\sum_{l \mid m} \frac{\phi(3^{e+1}l)}2
=3^e\sum_{l \mid m}\phi(l)
=3^e m=n.
\end{align*}
\end{enumerate}
\end{proof}

\begin{cor}\label{gcd-T}
Let $f_n(x)=2T_n\left(\frac{x}2+b\right)-2b$.
For any positive integers $m, n\ (m<n)$, put $d=\gcd(m,n)$ and $m=dm_1, n=dn_1$.
Then we have
\[
\gcd(f_m(x), f_n(x))
=
f_d(x)
\]
in the following cases;
\begin{enumerate}[{\rm (1)}]
\item
$b=0, \gcd(m_1,2)=\gcd(n_1,2)=1$,

\item
$b=1$,
\item
$b=-1$, $\gcd(m_1,2)=\gcd(n_1,2)=1$,
\item
$b=\frac12$, $\gcd(m_1,6)=\gcd(n_1,6)=1$,
\item
$b=-\frac12$, $\gcd(m_1,3)=\gcd(n_1,3)=1.$
\end{enumerate}
Otherwise, $f_m(x)$ and $f_n(x)$ are coprime.
\end{cor}

\begin{proof}
We first prove that $f_m(x)$ and $f_n(x)$ are coprime in case $b>1$.
It is enough to show that $T_m(\frac x2)-b$ and $T_n(\frac x2)-b$ are coprime.
By Proposition \ref{root-T}, let $\alpha=\sqrt[m]{B}\zeta_m^i + \sqrt[m]{B}^{-1}\zeta_m^{-i}$ be a root of $T_m(\frac x2)-b$ where $B=b + \sqrt{b^2 -1 }$.
Suppose that $T_n(\frac {\alpha}2)-b=0$ for some $n>m$.
By Lemma \ref{lem},  we obtain that
\begin{align*}
T_n\left(\frac {\alpha}2\right)-b=0
&\iff
(\sqrt[m]{B}\zeta_m^i)^n + (\sqrt[m]{B}\zeta_m^{i})^{-n}-2b=0\\
&\iff
(\sqrt[m]{B}\zeta_m^i)^n=B \ \text{or}\ B^{-1}
\end{align*}
If $(\sqrt[m]{B}\zeta_m^i)^n=B$ (resp. $(\sqrt[m]{B}\zeta_m^i)^n=B^{-1}$), then we have
$B^n=B^m
$
(resp. $B^n=B^{-m}$)
and thus $B$ is a roots of unity.
But it is impossible because $B=b+\sqrt{b^2-1}>b>1$.
So $T_m(\frac x2)-b$ and $T_n(\frac x2)-b$ have no common root, and hence they are coprime.
The same argument works in case $b<-1$.

We next prove (4) in case $b=\frac12$. The other cases are proven by the same argument.

By the factorization of $f_n(x)$ (Theorem \ref{factorization-T}), we obtain that
\begin{align*}
\gcd(f_m(x), f_n(x))
&=\prod_{\substack{k \mid m,\ k\mid n \\ \gcd(m/k, 6)=\gcd(n/k, 6) = 1}}\Psi_{6k}(x+1)\\
&=\prod_{\substack{k \mid d \\ \gcd(m_1d/k , 6)=\gcd(n_1d/k, 6) = 1}}\Psi_{6k}(x+1).
\end{align*}
If $\gcd(m_1,6)=\gcd(n_1,6)=1$ then we obtain that for $k\mid d$,
\[
\gcd\left(m_1 d/k, 6\right)=\gcd\left(n_1d/k, 6\right) = 1 
\iff
\gcd\left(d/k, 6\right)=1.
\]
Thus we have
\begin{align*}
\gcd(f_m(x), f_n(x))
&=\prod_{\substack{k \mid d \\ \gcd(m_1d/k, 6)=\gcd(n_1d/k, 6) = 1}}\Psi_{6k}(x+1)\\
&=\prod_{\substack{k \mid d \\ \gcd(d/k, 6)= 1}}\Psi_{6k}(x+1)
=f_d(x).
\end{align*}

If one of $m_1, n_1$ is not prime to $6$ then the condition $\gcd(m_1d/k, 6)=\gcd(n_1d/k, 6) = 1$ can not hold.
Thus in this case $f_m(x)$ and $f_n(x)$ are coprime.
\end{proof}

\begin{rmk}
In case $b=0$, Theorem \ref{factorization-T} and Collary \ref{gcd-T} have essentially proved in \cite{H} and \cite{RTW} respectively.
\end{rmk}

%%%%%%%%%%%%%%%%%%%%%%%%%%%%%%%%%%%%
%%                       Chains of monomial type
%%%%%%%%%%%%%%%%%%%%%%%%%%%%%%%%%%%%

\section{Chains of monomial type}

We here discuss conditions in Definition \ref{conditions} for a chain $\{f_n(x)\}_{n \in \N}$ of monomial type in $\Q[x]$; 
$$
f_n(x) = \frac{1}{a} ((ax + b)^n - b)\qquad (a, b\in \Q, a\neq 0).
$$

%%%%%%%%%%%%%%%%%%%%%%%%%%%%%%%%%%%%
%          Monic and integer coefficient conditions
%%%%%%%%%%%%%%%%%%%%%%%%%%%%%%%%%%%%

\subsection{Monic and integer coefficient conditions}

\begin{prop}\label{Z[x]}
Let $\{f_n(x)\}_{n \in \N}$ be a chain of monomial type in $\Q[x]$.
Then we have
\begin{enumerate}[{\rm (1)}]
\item
\ref{itm:i} holds if and only if $a = 1$.

\item \ref{itm:ii} holds if and only if $a, b \in \Z$ and  $a\mid b(b - 1).$ 
\end{enumerate}
\end{prop}

\begin{proof}
(1)
It is obvious since the leading coefficient of $f_n(x)$ is $a^{n-1}$ for all $n \in \N$.  

(2)
Since $f_2(x) = ax^2 + 2bx + \frac{b^2 - b}{a}$ and $f_3(x) = a^2x^3 + 3abx^2 + 3b^2x + \frac{b^3- b}{a}$, it is easily seen that
\begin{align*}
f_2(x), f_3(x)\in \Z[x] \ \iff \ a, b \in \Z \ \text{and}\ a \mid b(b-1).
\end{align*}
On the other hand, for any $n\geq 2$, we have
\[
f_n(x)=(ax+b)f_{n-1}(x)+bx+\frac{b^2-b}a.
\]
Thus, by induction on $n$, the desired statement is proved.
\end{proof}

%%%%%%%%%%%%%%%%%%%%%%%%%%%%%%%%%%%%
%          Euclidean Division
%%%%%%%%%%%%%%%%%%%%%%%%%%%%%%%%%%%%

\subsection{Euclidean Division}

\begin{prop}\label{E-Div-mono}
For $a, b\in \Q\, (a\neq 0)$, let $f_n(x) = \frac{1}{a} ((ax + b)^n - b)$.
For any $m, n\, (m\leq n)$, the quotient $q_{m,n}(x)$ and the remainder $r_{m,n}(x)$ of the Euclidean division of $f_n(x)$ by $f_m(x)$ are given by
\begin{align*}
q_{m,n}(x)=a\sum_{i=1}^kb^{i-1}f_{n-im}(x)+\sum_{i=1}^kb^i, \quad 
r_{m,n}(x)=b^kf_{n-km}(x)+\frac{b^{k+1}-b}a,
\end{align*}
where $k=[\frac{n}{m}]$.
\end{prop}
\begin{proof}
Let $g_n(x)=x^n-b$ and the quotient $Q_{m,n}(x)$ and the remainder $R_{m,n}(x)$ of the Euclidean division of $g_n(x)$ by $g_m(x)$.
Then we have
\[
q_{m,n}(x)=Q_{m,n}(ax+b),\quad r_{m,n}(x)=\frac1a R_{m,n}(ax+b)
\]
since $f_n(x)=\frac1a g_n(ax+b)$ and $g_n(x)=af_n(\frac{x-b}a)$.
By an elementary computation, we have 
$$Q_{m,n}=\sum_{i=1}^kb^{i-1}x^{n-im},\quad R_{m,n}(x)=b^kx^{n-km}-b,$$
which completes the proof.
\end{proof}

\begin{cor}\label{Div-mono}
Let $f_n(x)$ be as in Proposition \ref{E-Div-mono}.
If $b\neq 0,\pm1$, then $f_m(x)$ is not a divisor of $f_n(x)$ for any $m<n.$
In case $b=0, \pm 1$, a condition that $f_m(x)$ is a divisor of $f_n(x)$ is equivalent to
the following
\begin{enumerate}[{\rm (1)}]
\item
in case $b=0$, $m\leq n$.
\item
in case $b=1$, $m\mid n$.
\item
in case $b=-1$, $m\mid n$ and $\gcd(n/m, 2)=1$.
\end{enumerate}
\end{cor}

\begin{proof}
The condition that $f_m(x)$ is a divisor of $f_n(x)$ is equivalent to that the remainder $r_{m,n}(x)$ is zero.
By Proposition \ref{E-Div-mono}, the statement in case $b=0$ is clear, and
in case $b\neq 0$, we have
\[
r_{m,n}(x)=0 \iff n=km \ \text{and}\ b^k-b=0.
\]
Since $b\in \Q$, the equation $b^k-b=0$ for $k>1$ implies that $b=0, \pm1$.
This shows the statement in case $b\neq 0, \pm 1.$
And the statements in case $b=\pm 1$ is now clear, and the proof completes. 
\end{proof}

%%%%%%%%%%%%%%%%%%%%%%%%%%%%%%%%%%%%
%          GCD
%%%%%%%%%%%%%%%%%%%%%%%%%%%%%%%%%%%%

\subsection{Greatest common divisor and Factorization}

From Proposition \ref{Z[x]}, a chain of monic polynomials of monomial type in $\Z[x]$ is given by
\[
\{(x+b)^n-b\}_{n\in \N}\quad (b\in \Z).
\]
We omit a discussion for $b=0$ below.

The following factorizations in $\Z[x]$ are well known 
\begin{equation*}
x^n-1 = \prod_{k \mid n} \Phi_k(x),\quad x^n+1 = \prod_{\substack{k\mid 2n \\ {\rm gcd} (2n/k, 2) = 1}} \Phi_k(x).
\end{equation*}
Here $\Phi_k(x)$ is $k$-th cyclotomic polynomial, i.e. the minimal polynomial of a  primitive $k$-th root of unity.
Thus we have the factorizations 
\begin{align}\label{factrization-mono}
(x+1)^n-1=\prod_{k \mid n} \Phi_k(x+1),\quad (x-1)^n+1=\prod_{\substack{k\mid 2n \\ {\rm gcd} (2n/k, 2) = 1}} \Phi_k(x-1).
\end{align}

\begin{cor}\label{gcd-mono}
Let $f_n(x)=(x+b)^n-b\ (b\neq 0)$.
For any positive integers $m, n\ (m<n)$, put $d=\gcd(m,n)$ and $m=dm_1, n=dn_1$.
Then we have
\[
\gcd(f_m(x), f_n(x))
=
f_d(x)
\]
in the following cases;
\begin{enumerate}[{\rm (1)}]
\item
$b=1$,

\item
$b=-1$, $\gcd(m_1,2)=\gcd(n_1,2)=1.$
\end{enumerate}
Otherwise, $f_m(x)$ and $f_n(x)$ are coprime.
\end{cor}

\begin{proof}
In cases (1) and (2), the assertion follows from the factorizations in \eqref{factrization-mono}.
In case $b=-1$, if one of $m_1, n_1$ is even, there is no $k$ such that
\[
k\mid 2n,\ k\mid 2m,\  \gcd(2n/k,2)=\gcd(2m/k,2)=1.
\]
Thus $f_m(x)$ and $f_n(x)$ are coprime.

Next we assume $b>1$.
Then the roots of $f_m(x)$ are given by
\[
-b+\sqrt[m]{b}\zeta_m^j\quad (j=0,1,\dots,m-1).
\]
If $f_n(-b+\sqrt[m]{b}\zeta_m^j)=0$ for some $m<n$, then we have
\begin{align*}
f_n(-b+\sqrt[m]{b}\zeta_m^j)=0
\iff 
b^{\frac nm}\zeta_m^{jn}=b
\ \Longrightarrow
b^{n-m}=1
\ \Longrightarrow
b=\pm1
\end{align*}
which contradicts to the assumption $b>1$.
Thus $f_m(x)$ and $f_n(x)$ has no common roots, and hence they are coprime.
The same argument works in case $b<-1$, and the proof complets.
\end{proof}

%%%%%%%%%%%%%%%%%%%%%%%%%%%%%%%%%%%%
%          Consequence
%%%%%%%%%%%%%%%%%%%%%%%%%%%%%%%%%%%%
%
%
%
%
%\section{A consequence}
%
%As a consequence of all results in \S3 and \S4, we have the following.
%\begin{cor}[Corollary \ref{intro-consequence}]\label{consequence}
%Let $\{f_n(x)\}_{n \in \N}$ be a chain of Chebyshev (resp. monomial ) type in $\Q[x]$. Then $\{f_n(x)\}_{n \in \N}$ satisfies all of \ref{itm:i}, \ref{itm:ii} and \ref{itm:iii} if and only if $(a,b)=(\frac12,1)$ (resp. $(a,b) = (1,1)$), that is,
%\[
%f_n(x)=2T_n\left(\frac{x}2+1\right)-2\quad ( \text{resp.}\ f_n(x)=(x+1)^n-1\ ).
%\]
%Additionally, in this case \ref{itm:iv} is also satisfied. 
%\end{cor}

%Corollary \ref{consequence} shows the algebraic particularity of the chains $\{F_n(x)\}_{n\in \N}$ and $\{\tilde{F}_n(x)\}_{n\in \N}$ of \eqref{Fn} in all chains.
%Since both polynomials come from the weighted sums for cycle graphs with pendant edges, it is natural to ask a graph theoretical explanation of the division properties and the commutativity under composition.
%Such graph theoretical approach to these properties will be developed in a subsequent paper. 

\section{Chains over a field of positive characteristic}
Let $k$ be a field of characteristic $p>0$.
In positive characteristic, commuting polynomials have been studied (cf. \cite{Sch}). 
We here discuss chains in  $k[x]$.

\subsection{Classification theorem}
We discuss Theorem \ref{class-chain} for chains in $k[x]$ and prove the following.

\begin{thm}[Theorem \ref{intro-thm-A}]\label{thm-A}
%Assume that $p\geq 3$.
Let $\{f_n(x)\}_{n \in \N}$ be a chain in $k[x]$.
Then $\{f_n(x)\}_{n \in \N}$ is similar to only one of the two $\{ x^n\}_{n\in\N}$ or $\{G_n(x)\}_{n\in\N}$.
Here $G_n(x):=F_n(x) \pmod p$.
\end{thm}

\subsubsection{Proof in case $p\neq 2$}
We prove by a similar strategy of the proof of Theorem \ref{class-chain} in \cite{W}.

In case $p\geq 5$, the same argument in \cite[\S4]{W} works and thus we omit this case.

In case $p=3$.
We write $f_2(x)=a_2x^2+a_1x+a_0$.
By a similarity via $\lambda(x)=\frac{x}{a_2}-\frac{a_1}{2a_2}$, we may assume that $f_2(x)=x^2+c$.
Since $\{f_n(x)\}_{n\in \N}$ is a chain, $f_2(x)$ is commute under composition with
\[
f_5(x)=b_5x^5+b_4x^4+b_3x^3+b_2x^2+b_1x+b_0.
\]
Thus we have
\begin{align}\label{eq-A1}
f_5(x)^2+c=f_5(x^2+c)\iff f_5(x)^2=f_5(x^2+c)-c,
\end{align}
and hence we get 
\[
f_5(x)^2=f_5(-x)^2\iff \bigl(f_5(x)-f_5(-x)\bigr)\bigl(f_5(x)+f_5(-x)\bigr)=0.
\]
Since $b_5\neq 0$ and thus $f_5(x)-f_5(-x)=2b_5x^5+2b_3x^3+2b_1x\neq 0$, we have $f_5(x)=-f_5(-x)$ which implies that $f_5(x)=b_5x^5+b_3x^3+b_1x$.
By substituting this presentation of $f_5(x)$ into equation \eqref{eq-A1}, we get
\begin{align*}
b_5x^{10}+2b_5cx^8+(b_5c^2+b_3)x^6+b_5c^3x^4+(2b_5c^4+b_1)x^2+b_5c^6+b_3c^3+b_1c\\
=b_5^2x^{10}+2b_5b_3x^8+(b_3^2+2b_5b_1)x^6+2b_3b_1x^4+b_1^2x^2+c
\end{align*}
Comparing the coefficients of $x^{10}, x^8$ and $x^6$, we have
\[
b_5=1, \quad b_3=c, \quad b_1=-b_3=-c.
\]
Furthermore, comparing the coefficients of $x^4$ shows that 
\[
b_5c^3=2b_3b_1\iff c^3=c^2\iff c=0, 1.
\]
\begin{enumerate}
\item
$c=0$ case.
$f_2(x)=x^2$.
By Proposition \ref{prop-A}, $\{f_n(x)\}_{n\in \N}$ is similar to $\{x^n\}_{n\in \N}$.
\item
$c=1$ case.
$f_2(x)=x^2+1.$
Consider $\lambda(x)=x-1$.
We have
\[
(\lambda^{-1}\circ f_2\circ \lambda)(x)=x^2-2x\equiv G_2(x)= F_2(x)\pmod 3.
\]
By Proposition \ref{prop-A} below, $\{f_n(x)\}_{n\in \N}$ is similar to $\{G_n(x)\}_{n\in \N}$.
\end{enumerate}

\begin{prop}[\text{\cite[Theorem 3.2]{W}}]\label{prop-A}
Assume $p\geq 3$.
There is at most one polynomial of degree $k\geq 1$ that commutes with a given polynomial of degree 2.
\end{prop}

\begin{proof}
Let $f(x)$ be a polynomial of degree 2.
By a similarity, we may assume that $f(x)=x^2+a$.

Now if we have two distinct polynomials $g_1(x), g_2(x)$ of degree $k\geq 1$ that commutes with $f(x)$, then we have
\[
g_i(x^2+a)=(g_i\circ f)(x)=(f\circ g_i)(x)=g_i(x)^2+a\quad (i=1,2).
\]
This equality shows that $g_i(x)$ is monic.
Therefore $r(x):=g_1(x)-g_2(x)$ is a non-zero polynomial of degree $t<k$.
On the other hand, we have
\begin{align*}
r(x^2+a)
&=
(r\circ f)(x)\\
&=g_1(x^2+a)-g_2(x^2+a)\\
&=g_1(x)^2-g_2(x)^2\\
&=r(x)(g_1(x)+g_2(x)).
\end{align*}
By comparing the degree of both side, we see that $2t=t+k$, or $t=k$ which is a contradiction.
\end{proof}

\subsubsection{Proof in case $p=2$}

\begin{rmk}
In case $p=2$, Proposition \ref{prop-A} does not hold.
For example, the following 8 polynomials of degree 3 commute with $f(x)=x^2$;
\begin{enumerate}
\item
$g_1(x)=x^3$, $h_1(x)=x^3+x^2+x$,
\item
$g_2(x)=x^3+1$,  $h_2(x)=x^3+x^2+x+1$,
\item
$g_3(x)=x^3+x^2$, $h_3(x)=x^3+x+1$,
\item
$g_4(x)=x^3+x$,  $h_4(x)=x^3+x^2+1$.
\end{enumerate}
For each $i=1,\dots,4$, $g_i(x)$ is similar to $h_i(x)$ via $\lambda(x)=x+1$.
\end{rmk}

\begin{proof}[Proof of Theorem \ref{thm-A}]
We write $f_2(x)=a_2x^2+a_1x+a_0$.
By a similarity via $\lambda(x)=\frac{x}{a_2}$, we may assume that $f_2(x)=x^2+a_1x+a_0$.
Since $\{f_n(x)\}_{n\in \N}$ is a chain, $f_2(x)$ is commute under composition with
\begin{align*}
f_3(x)&=b_3x^3+b_2x^2+b_1x+b_0.
\end{align*}
Comparing the coefficients of $(f_2\circ f_3)(x)$ and $(f_3\circ f_2)(x)$;
\begin{align*}
(f_2\circ f_3)(x)&=b_3^2x^6+b_2^2x^4+a_1b_3x^3+(b_1^2+a_1b_2)x^2+a_1b_1x+b_0^2+a_1b_0+a_0,\\
(f_3\circ f_2)(x)&=b_3x^6+b_3a_1x^5+(b_3a_0+b_3a_1^2+b_2)x^4+b_3a_1^3x^3\\
&\qquad +(b_3a_1^2a_0+b_3a_0^2+b_2a_1^2+b_1)x^2+(b_3a_0^2a_1+b_1a_1)x+b_3a_0^3\\
&\qquad \qquad  +b_2a_0^2+b_1a_0+b_0,
\end{align*}
we see that 
\[
b_3=1,\ a_1=0,\ b_2^2+b_2=a_0,\ b_1^2+b_1=a_0^2,\ b_0^2+a_0=f_3(a_0).
\]
Thus we have $f_2(x)=x^2+b_2^2+b_2=(x+b_2)^2+b_2$.
By a similarity via $\lambda(x)=x+b_2$, we may assume that $f_2(x)=x^2$ and $f_3(x)=x^3+b_1x+b_0$.
In this situation, the equality $(f_2\circ f_3)(x)=(f_3\circ f_2)(x)$ gives
\begin{align}\label{f_3-coef}
b_1^2=b_1,\ b_0^2=b_0.
\end{align}
The similar computation for $f_2(x)=x^2$ and $f_5(x)=c_5x^5+c_4x^4+c_3x^3+c_2x^2+c_1x+c_0$ gives
\begin{align}\label{f_5-coef}
c_5=1,\ c_i^2=c_i\ (i=0,1,2,3,4).
\end{align}

Since $f_3(x)$ also commute under composition with $f_5(x)$, comparing the coefficients of $x^{14}, 
\dots, x^{10}$ in $(f_3\circ f_5)(x)=(f_5\circ f_3)(x)$;
\begin{align}\label{f_3-f_5}
(f_3\circ f_5)(x)
=&x^{15}
+c_4x^{14}
+(c_3+c_4^2)x^{13}
+(c_2+c_4^3)x^{12}
+(c_1+c_3c_4^2+c_3^2)x^{11}\\\notag
&+(c_0+c_2c_4^2+c_3^2c_4)x^{10}
+(c_1c_4^2+c_3^3+c_2^2)x^9%\\\notag
% &+(c_0c_4+b_2c_4+c_2c_3+c_2c_4)x^8
% +(c_1c_3+c_2c_3+c_1)x^7
+\cdots ,\\\notag
(f_5\circ f_3)(x)
=&x^{15}
+b_1x^{13}
+(b_0+c_4)x^{12}
+c_3x^9+\cdots ,
\end{align}
we see that
\begin{align}\label{b-c}
c_4=0,\ c_3=b_1,\ c_2=b_0=0,\ c_1=b_1,\ c_0=0.
\end{align}
Here we used \eqref{f_5-coef}.
By \eqref{f_3-coef}, we now have
\begin{enumerate}
\item 
$b_1=0$ case.
\[
f_3(x)=x^3,\quad f_5(x)=x^5.
\]
\item
$b_1=1$ case.
\begin{align*}
    f_3(x)&=x^3+x=x(x+1)^2=G_3(x),\\ f_5(x)&=x^5+x^3+x=x(x^2+x+1)^2=G_5(x).
\end{align*}

\end{enumerate}
Thus, up to similarity there are two possibilities of chains of degree $\leq 5$.
The assertion now follows from Proposition \ref{prop-A2} below.
\end{proof}

\begin{prop}\label{prop-A2}
There is at most one monic polynomial of degree $k\geq 1$ that commutes with a given monic polynomial of degree 3.
\end{prop}

\begin{proof}
Let $f(x)=x^3+a_2x^2+a_1x+a_0$ be a monic polynomial of degree 3.
If we have two distinct monic polynomials $g_1(x), g_2(x)$ of degree $k\geq 1$ that commute with $f(x)$.
Assume $r(x):=g_1(x)+g_2(x)$ is a non-zero polynomial of degree $t<k$.
Since we work over characteristic $2$, we have $g_2(x)=g_1(x)+r(x)$.
On the other hand, we have
\begin{align*}
(r\circ f)(x)
&=(g_1\circ f)(x)+(g_2\circ f)(x)\\
&=f(g_1(x))+f(g_2(x))\\
&=r(x)(g_1(x)^2+g_1(x)r(x)+r(x)^2+a_2r(x)+a_1).
\end{align*}
By comparing the degree of both sides, we see that $3t=t+2k$, or $t=k$ which is a contradiction.
\end{proof}

\subsection{Factorization of $F_n(x)$ and $\tilde{F}_n(x)$ modulo $p$}
From Theorem \ref{factorization-T}(2) and \eqref{factrization-mono}, we have
\[
F_n(x)=c_n(x)\prod_{\substack{k \mid n,\ 2 < k}} \Psi_k(x+2)^2
,\quad 
\tilde{F}_n(x)=(x+1)^n-1 = \prod_{k \mid n} \Phi_k(x+1)
\]
Here $\Psi_k(x)$ is the minimal polynomial of $2\cos(\frac{2\pi}k)$, and $\Phi_k(x)$ is the $k$-th cyclotomic polynomial.
And $c_n(x)$ is defined as follows (see Definition \ref{def-c});
\begin{align*}
c_n(x)&:=
\begin{cases}
x& (n : odd)\\
x(x + 4) & (n : even).
\end{cases}
\end{align*}
Therefore, to give the factorization of $F_n(x), \tilde{F}_n(x)\pmod p$, it is enough to give the factorization of $\Psi_n(x+2), \Phi_n(x+1) \pmod p$ for any $n$.
For cyclotomic polynomials, the following theorem is known.

\begin{thm}[\cite{G}]\label{fact-cycl}
Let $p$ be a prime, and let $n$ be a positive integer copirme to $p$.
Let $f$ be the (multiplicative) order of $p\pmod n$.
Denote by $\phi$ the Euler's totient function.
\begin{enumerate}[{\rm (1)}]
\item
$\Phi_n(x)\pmod p$ factors into a product of $\phi(n)/f$ distinct irreducible polynomials each of degree $f$. 
\item
For any integer $r$, $\Phi_{p^rn}(x)\equiv \Phi_n(x)^{\phi(p^r)}\pmod p$.
\end{enumerate}
\end{thm}

We derive the factorization for $\Psi_n(x)$ from Theorem \ref{fact-cycl}.

\begin{prop}\label{factor-G}
Let $p$ be a prime and let $n$ be a positive integer copirme to $p$.
Let $f$ be the (multiplictive) order of $p\pmod n$.
\begin{enumerate}[{\rm (1)}]
\item
If $p^{f/2}\equiv -1\pmod n$, then $\Psi_n(x)\pmod p$ factors into a product of $\phi(n)/f$ distinct irreducible polynomials each of degree $f/2$. 
\item
If $f$ is odd or $p^{f/2}\not \equiv -1\pmod n$, then $\Psi_n(x)\pmod p$ factors into a product of $\phi(n)/2f$ distinct irreducible polynomials each of degree $f$. 

\item
For any integer $r$, $\Psi_{p^rn}(x)\equiv \Psi_n(x)^{\phi(p^r)}\pmod p$.
\end{enumerate}
\end{prop}

\begin{proof}
Put $X:=x+x^{-1}$.
For any positive integer $n$, we have 
\begin{align}\label{eq-psi-phi}
x^{-\frac{\phi(n)}2}\Phi_n(x)=\Psi_n(X).
\end{align}
Let $\zeta\in \F_{p^f}$ be a primitive $n$-th root of unity.
A factor of $\Phi_n(x)\pmod p$ is given by
\[
h_\zeta(x):=(x-\zeta)(x-\zeta^p)\cdots (x-\zeta^{p^{f-1}}).
\]
\begin{enumerate}[{\rm (1)}]
\item
By the assumption $p^{f/2}\equiv -1\pmod n$, we have $h_\zeta(x)=h_{\zeta^{-1}}(x)$, and hence the minimal polynomial of $\zeta+\zeta^{-1}$ over $\F_p$ is
\begin{align*}
x^{-\frac f2}h_\zeta(x)
&=(x^2-(\zeta+\zeta^{-1})x+1)\cdots (x^2-(\zeta^{p^{\frac f2-1}}+\zeta^{p^{f-1}})x+1)\\
&=(X-(\zeta+\zeta^{-1}))\cdots (X-(\zeta^{p^{\frac f2-1}}+\zeta^{p^{f-1}})).
\end{align*}
Thus $\Psi_n(X)\pmod p$ has an irreducible factor of degree $f/2$, and the assertion now follows from Theorem \ref{fact-cycl}(1) and \eqref{eq-psi-phi}.
\item
By the assumption, we have $h_\zeta(x)\neq h_{\zeta^{-1}}(x)$, and hence the minimal polynomial of $\zeta+\zeta^{-1}$ over $\F_p$ is
\begin{align*}
x^{-f}h_\zeta(x)h_{\zeta^{-1}}(x)
&=(x^2-(\zeta+\zeta^{-1})x+1)\cdots (x^2-(\zeta^{p^{f-1}}+\zeta^{-p^{f-1}})x+1)\\
&=(X-(\zeta+\zeta^{-1}))\cdots (X-(\zeta^{p^{f-1}}+\zeta^{-p^{f-1}})).
\end{align*}
Thus $\Psi_n(X)\pmod p$ has an irreducible factor of degree $f$, and the assertion now follows from Theorem \ref{fact-cycl}(1) and \eqref{eq-psi-phi}.

\item
From Theorem \ref{fact-cycl}(3) and \eqref{eq-psi-phi}, we obtain that 
\[
\Psi_{p^rn}(X)=x^{-\frac{\phi(p^rn)}2}\Phi_{p^rn}(x)=\left(x^{-\frac{\phi(n)}2}\Phi_n(x)\right)^{\phi(p^r)}=\Psi_n(x)^{\phi(p^r)}\pmod p.
\]
\end{enumerate}
\end{proof}

\begin{cor}[Corollary \ref{intro-cor-A}]\label{cor-A}
Let $p$ be a prime.
For any positive integer $r$, we have
\[
F_{p^r}(x)\equiv \tilde{F}_{p^r}(x)\equiv x^{p^r}\pmod p.
\]
\end{cor}

\begin{proof}
By Theorem \ref{factorization-T}(2), \eqref{factrization-mono}, Theorem \ref{fact-cycl}(3) and Proposition \ref{factor-G}(3), we have
\[
F_{p^r}(x)\equiv x \prod_{i=1}^r\Psi_1(x+2)^{\phi(p^i)}=x^{p^r}\pmod p,
\]
\[
\tilde{F}_{p^r}(x)\equiv x\prod_{i=1}^r\Phi_1(x+1)^{\phi(p^i)}=x^{p^r}\pmod p.
\]
\end{proof}

Combine Theorem \ref{thm-A} and Corollary \ref{cor-A}, we see that for any chain $\{f_n(x)\}_{n\in \N}$ over a filed $k$ of characteristic $p>0$, $f_{p^r}(x)$ is similar to $x^{p^r}$.

%%%%%%%%%%%%%%%%%%%%%%%%%%%%%%%%%%%
%
%      References
%
%%%%%%%%%%%%%%%%%%%%%%%%%%%%%%%%%%%

\end{document}